\documentclass[journal]{IEEEtran}

\usepackage[english]{babel}
\usepackage[utf8]{inputenc}
\usepackage{amsmath}
\usepackage{amsthm}
\usepackage{mathtools}
\usepackage{graphicx}
\usepackage{float}
\usepackage[justification=centering]{caption}
\usepackage{dsfont}

\newtheorem{theorem}{Theorem}
\newtheorem{definition}{Definition}
\newtheorem{assumption}{Assumption}

\title{Safely Optimizing Highway Traffic with Robust Model Predictive Control-based Cooperative Adaptive Cruise Control}
\author{Carlos M. Massera, ~\IEEEmembership{Student Member,~IEEE,} Marco H. Terra, ~\IEEEmembership{Member,~IEEE,} Denis F. Wolf, ~\IEEEmembership{Member,~IEEE,}%
\thanks{Carlos M. Massera and Denis F. Wolf are with the Institute of Mathematics and Computer Science of University of São Paulo.
\{massera, denis \}@icmc.usp.br}%
\thanks{Marco H Terra is with the São Carlos School of Engineering of University of São Paulo.
terra@sc.usp.br}
\thanks{This work was funded by FAPESP on the grant 2013/24542-7.}}%

\begin{document}
\maketitle

\begin{abstract}
Road traffic crashes have been the leading cause of death among young people. Most of these accidents occur when the driver becomes distracted due to fatigue or external factors. Vehicle platooning systems such as Cooperative Adaptive Cruise Control (CACC) are one of the results of the effort devoted to the development of technologies for decreasing the number of road crashes and fatalities. Previous studies have suggested such systems improve up to 273\% highway traffic throughput and fuel consumption in more than 15\% if the clearance between vehicles in this class of roads can be reduced to 2 meters. This paper proposes an approach that guarantees a minimum safety distance between vehicles taking into account the overall system delays and braking capacity of each vehicle. A $l\infty$-norm Robust Model Predictive Controller (RMPC) is developed to guarantee the minimum safety distance is not violated due to uncertainties on the lead vehicle behavior. A formulation for a lower bound clearance of vehicles inside a platoon is also proposed. Simulation results show the performance of the proposed approach compared to a nominal controller when the system is subject to both modeled and unmodeled disturbances.
\end{abstract}

\section{Introduction} 

Road traffic crashes are the leading cause of death among young people between 10 and 24 years old \cite{world2007youth}. Most of these accidents occur when the driver is unable to maintain the vehicle control due to fatigue or external factors \cite{national2008national}. In recent years, both academia and industry have been devoted towards the development of safety systems for decreasing the number of road accidents. Vehicle platooning systems, such as Adaptive Cruise Control (ACC) \cite{1622978} and Cooperative Adaptive Cruise Control (CACC) \cite{6246707}, are examples of this class of systems through which a vehicle regulates its own speed and distance to the vehicle ahead based on a sensor suite and wireless communication interfaces, respectively.
 
ACC systems have already reached consumer market through radar \cite{1622978} and camera \cite{gat2005monocular} technologies. Such interest in vehicle platooning has led to the development of vehicle-specific communication protocols, such as 802.11p \cite{jiang2008ieee} and DSRC \cite{xu2004vehicle} necessary for the development of CACC systems. Studies have shown that CACC can improve up to 273\% road traffic throughput and significant reductions to fuel consumption \cite{tsugawa2014results} given a high market penetration \cite{6093130, van2006impact, ngoduy2013instability}. However, such improvements rely on the assumption CACC technology will significantly reduce the clearance between vehicles. In this proposed scenario, any disturbance in the lead vehicle behavior could lead to an accident. Therefore, stability, performance robustness and robustness against violating the minimum distance between vehicles must be ensured.

This paper proposes a cooperative adaptive cruise controller based on $l\infty$-norm robust model predictive controller capable of rejecting uncertainties on the lead vehicle acceleration behavior and operate close to the minimum safety distance between vehicles. The main contributions of this work are the formulation of such CACC system and a novel formulation of the minimum safety distance between vehicles that incorporates the lead vehicle braking capacity.

The remainder of the paper is organized as follows: Section \ref{sec_related} discusses related work; Section \ref{sec_model} presents the system modeling; Section \ref{sec_constraints} describes the applications safety and comfort constraints; Section \ref{sec_controller} addresses the proposed controller formulation; Section \ref{sec_experiments} reports the experiments performed and Section \ref{sec_conclusion} provides the final remarks.

\begin{figure*}[ht]
\centering
\includegraphics[width=\textwidth]{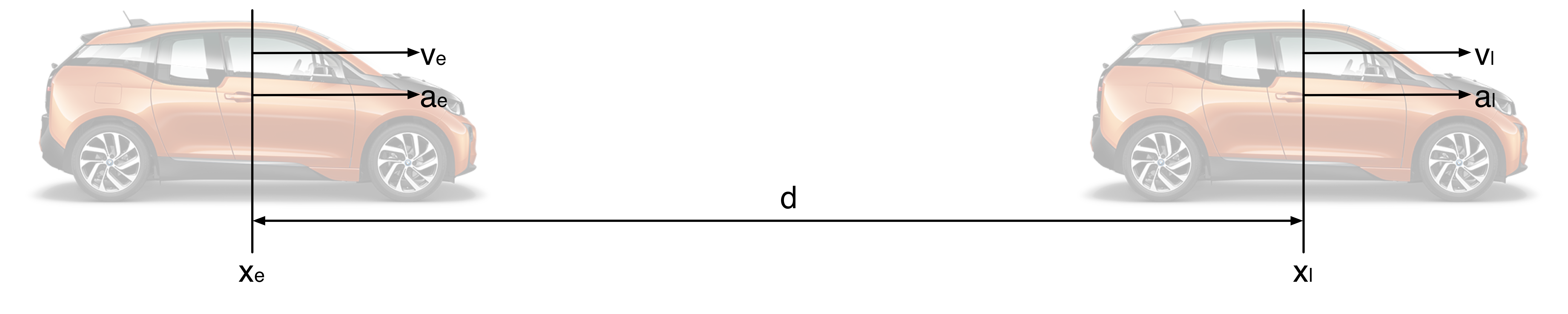}
\caption{Representation of the ego and lead vehicle models}
\label{fig_system_model}
\end{figure*}

\section{Related Work}
\label{sec_related}

During the 1980s and 1990s, both public and private programs started funding research in intelligent transportation systems based on fully or partially automated vehicle systems. The main focus was on improvements for highway throughput and vehicle safety through the creation of Automated Highway Systems (AHS) \cite{bender1969study}.

AHS research was nearly abandoned in the 2000s when studies changed focus from infrastructure-dependent to infrastructure-independent automation, which resulted in the current intelligent and autonomous vehicle research scenario. Commercially available systems, such as Adaptive Cruise Control (ACC) \cite{1622978} and Forward Collision Warning (FCW) \cite{gat2005monocular} have focused mostly on comfort, while the safety of the vehicle is still the driver's responsibility.

The development of systems based on vehicle-to-vehicle (V2V) and vehicle-to-infrastructure (V2I) communications have drawn growing interest from the research community due to improvements in communication bandwidth and processing power. Studies on automated infrastructures reappeared and culminated in the Grand Cooperative Driving Challenge (GCDC) \cite{guvenc2012cooperative_int} held in 2011, in which vehicles had to perform highway platooning tasks using embedded sensors and communications among them \cite{6246707}. The Karlsruher Institut für Technologie (KIT) won the competition with the AnnieWay project. A Linear Quadratic Regulator (LQR)-based controller was designed using all preceding communicating vehicles \cite{6246707}. However, the approach proposed by this project focused only on the headway time disregarding other safety requirements.

Tsugawa described results from heavy truck CACC studies performed in closed tracks in Japan, in 2014 \cite{tsugawa2014results}. The author developed a lateral and longitudinal automated platoon of three heavy trucks and demonstrated fuel savings decreased when the clearance between vehicles on the platoon increased for the first two vehicles, while the fuel savings were constant for the third vehicle for clearances between 2 and 20 meters. Therefore, the technical possibilities of decreasing the distance while maintaining vehicle safety must be investigated, specially for a market adoption phase when most platoons will contain a small number of vehicles. Besides the technical factors, drivers and passengers must feel comfortable with small gaps \cite{nowakowski2010cooperative}. 

Stanger and del Re \cite{stanger2013model} proposed the use of Model Predictive Control (MPC) for the explicit incorporation of fuel consumption optimization in the controller design. They penalized the approximate fuel consumption based on a nonlinear static fuel consumption map. The MPC approach also enabled the use of explicit constraints to represent mechanical limits of the vehicle. Several other studies \cite{kreuzen2012cooperative, sancar2014mpc, lang2014predictive} have proposed other linear and nonlinear MPC approaches to account for several preceding vehicles speed profiles and to optimize fuel consumption.

Safety envelope boundaries are easily corrupted by disturbances when vehicles cruise with small clearances. Therefore, robust or stochastic techniques should be used for the avoidance of violation of safety restrictions. Robust constrained optimal control techniques were proposed by Corona et al. \cite{corona2006robust}, who studied the use of Robust Hybrid MPC to deal with uncertainties of the piece-wise linear approximation of the powertrain and brake dynamics. However, they did not take into account uncertainties related to the lead vehicle behavior.

Moser et al. \cite{moser2015cooperative} designed a stochastic model of the driver behavior to improve fuel efficiency. The authors developed a conditional Gaussian graphical model to represent the probability distribution of a vehicle behavior in a prediction horizon of 15s combined with a Stochastic Model Predictive Control approach.

\section{Modeling}
\label{sec_model}

\begin{assumption}
The powertrain and brake dynamics of the controlled vehicle can be approximated by an inversible steady-state time-invariant model with no introduction of significant uncertainties to the system.
\label{ass_powertrain}
\end{assumption}

Assumption \ref{ass_powertrain} was experimentally validated in \cite{massera2014longitudinal}. Therefore, he powertrain and brake dynamics are abstracted from the model, reducing its complexity while maintaining performance.

The proposed controller was designed to regulate the vehicle distance, speed and acceleration considering only the immediate preceding (lead) vehicle. Therefore, a simplified kinematic model of the vehicle's longitudinal dynamics is proposed, described in Definitions \ref{def_system_model_1} and \ref{def_system_model_2}.

\begin{definition}
Let $ p_i(t) $, $ v_i(t) $ and $ a_i(t) $ respectively denote the position, velocity and acceleration of vehicle $i$. The vehicle dynamics are represented by%
\begin{equation}
\begin{matrix}
\dot{p}_i(t) &=& v_i(t)\\
\dot{v}_i(t) &=& a_i(t)
\end{matrix}
\end{equation}%
as shown in both vehicles of Figure \ref{fig_system_model}.
\label{def_system_model_1}
\end{definition}

\begin{definition}
Consider two adjacent vehicles on a platoon. Let the lead (preceding) vehicle states be denoted by subscript $ l $, the ego (succeeding) vehicle be denoted by a subscript $ e $ and $ d(t) = p_l(t) - p_e(t) $ be the distance between these vehicles. The relative dynamics are%
\begin{equation}
\begin{matrix}
\dot{d}  &=& v_l(t)  - v_e(t) \\
\ddot{d} &=& a_l(t)  - a_e(t).
\end{matrix}
\label{eq_system_distance}
\end{equation}
\label{def_system_model_2}
\end{definition}

Let the system state as $ x(t) = [d(t), v_l(t), v_e(t)]^T $ and the control input as $ u(t) = a_e(t) $. The system from Definitions \ref{def_system_model_1} and \ref{def_system_model_2} can be written as a linear affine discrete system of the form%
\begin{equation}
x_{k+1} = F x_k + G u_k + h
\end{equation}%
where, given a sampling time $ T_s $,%
\begin{equation}
F = \begin{bmatrix}
1 & T_s & -T_s\\
0 & 1 & 0\\
0 & 0 & 1
\end{bmatrix},\;
G = \begin{bmatrix}
-0.5 T_s^2\\0\\T_s
\end{bmatrix},\;
h = \begin{bmatrix}
0\\T_s a_{l,k}\\0
\end{bmatrix}.
\end{equation}

\section{Platooning Safety and Comfort Considerations}
\label{sec_constraints}

This section is divided into three parts - the first defines the minimum safety distance between two vehicles, the second investigates its impact to small clearance assumptions and the third addresses the comfort and safety restrictions the system must satisfy.

\subsection{Minimum Safety Distance}

Most economical benefits from CACC systems are related to the assumption this technology will enable vehicles to drive very close to each other. However, a worst-case representation of the minimum safety distance between vehicles is required to avoid compromising system safety and provide a lower bound to vehicle distances.

Approaches for Platoon safety can be categorized as centralized and decentralized. The centralized approach assigns the task of ensuring the safety of the platoon to one of its vehicles, typically its leader. Vehicles inside a platoon can operate in an unsafe region with respect to other vehicles while still maintaining global safety. The decentralized approach distributes the task of ensuring safety of the platoon to each vehicle contained in it. Therefore, every vehicle must ensure its safety with respect to preceding vehicles or, at least, the immediately preceding vehicle.

The minimum safety distance concept proposed in this paper focuses on the decentralized approach, in which every vehicle is responsible to ensure its safety with respect to its immediate preceding vehicle.

\begin{assumption}
The lead vehicle will never decelerate more than its claimed maximum braking capacity.
\label{ass_preceding_vehicle}
\end{assumption}

Assumption \ref{ass_preceding_vehicle} cannot be satisfied if the lead vehicle is involved in a collision, and further investigation is required for the removal of the assumption.

\begin{figure}[ht]
\centering
\includegraphics[width=\columnwidth]{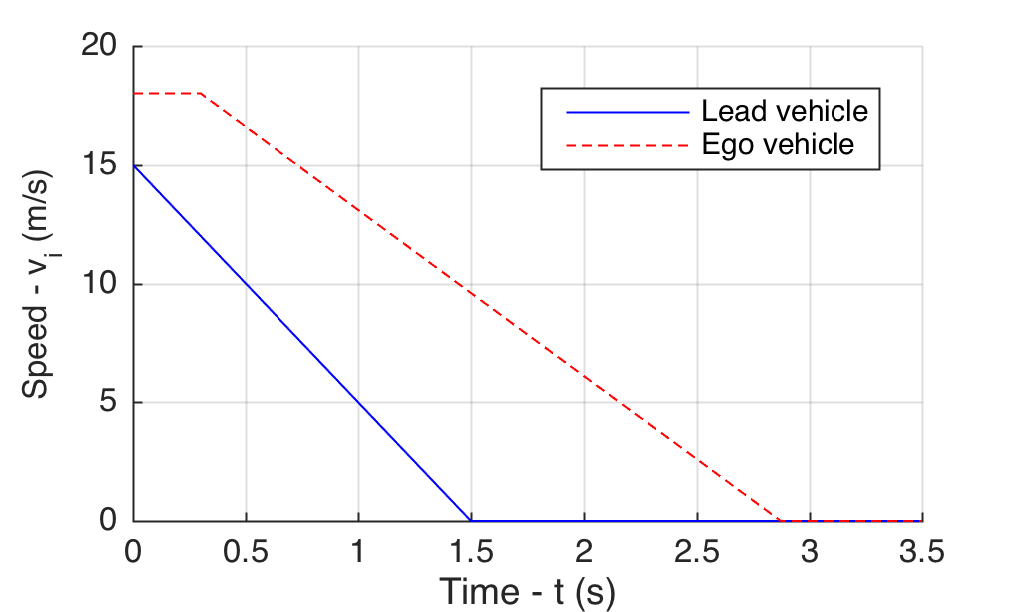}
\caption{Example of the vehicle speed profiles at an emergency situation of Definition \ref{def_min_safe_dist}, where the lead vehicle cruises at $15 m/s$ with a braking capacity $a_l^b = 10 m/s^2$ and an ego vehicle cruises at $18 m/s$ with a braking capacity $a_e^b = 7 m/s^2$}
\label{fig_safety_braking_speed}
\end{figure}

\begin{definition}
Let $ \phi $ be the sum of worst-case delays from communication, processing and actuation. Consider the system model from Definition \ref{def_system_model_2} where both the lead and ego vehicles respectively brake at their maximum braking capacity $ a_l^b $ and $ a_e^b $ at time $ t $ and $ t + \phi $ until both reach complete stop. $ d_{safe}(t) \in \mathds{R} $ is said to be the minimum safety distance if%
\begin{equation}
d(t) \ge d_{safe}(t) \Rightarrow \forall \epsilon \in \mathds{R}_{\ge 0}. \; d(t + \epsilon) \ge 0
\end{equation}%
with%
\begin{equation}
\dot{v}_l(\tau) = a_l(\tau) = \left\{ \begin{matrix}
- a_l^b &,& 0 \le \tau - t \le \frac{v_l(t)}{a_l^b}\\
0 &,& otherwise
\end{matrix} \right. ,
\end{equation}%
and%
\begin{equation}
\dot{v}_e(\tau) = a_e(\tau) = \left\{ \begin{matrix}
- a_e^b &,& 0 \le \tau - t - \psi \le \frac{v_e(t)}{a_e^b}\\
0 &,& otherwise.
\end{matrix} \right. .
\end{equation}%
Therefore, $ d_{safe} $ is given by%
\begin{equation}
\begin{matrix*}[l]
d_{safe} =& \underset{d}{\min} & d \\
& s.t. & d + \underset{t \le \epsilon \le \infty}{\min} \left\{ \int_t^{\epsilon} v_l(\tau) - v_e(\tau) d\tau \right\} \ge 0
\end{matrix*}
\label{eq_safety_distance_int}
\end{equation}
\label{def_min_safe_dist}
\end{definition}

Figure \ref{fig_safety_braking_speed} shows an example of the vehicles speed profiles for Definition \ref{def_min_safe_dist}.

\begin{theorem}
Consider Definition \ref{def_min_safe_dist}. $ d_{safe}$ has a closed form solution%
\begin{equation}
d_{safe} = \left\{ \begin{matrix*}[l]
d_{safe}^{lo} &, \epsilon \in [t + \phi, t + t_f^{min}) \\ 
\max\left(0, d_{safe}^{ub} \right)  &, otherwise 
\end{matrix*}\right.
\label{eq_minimum_safety_distance}
\end{equation}%
where%
\begin{equation}
\begin{matrix*}[l]
\epsilon^* &= \frac{v_e(t) - v_l(t) + a_e^b \phi}{a_e^b - a_l^b}\\
d_{safe}^{ub} &= v_e(t) \phi + \frac{v_e(t)^2}{a_e^b} - \frac{v_l(t)^2}{a_l^b}\\
d_{safe}^{lo} &= (v_e(t) - v_l(t)) \epsilon^* - (a_e^b - a_l^b) \frac{\epsilon^{*2}}{2} + a_e^b \frac{\phi^2}{2}.
\end{matrix*}
\end{equation}
\label{the_safe_distance}
\end{theorem}
\begin{proof}

Let $ t_f^e = \phi + v_e(t) / a_e^b $, $ t_f^l = v_l(t) / a_l^b $, $ t_f^{min} = \min\left(t_f^e, t_f^l\right) $ and $ t_f^{max} = \max\left(t_f^e, t_f^l\right) $ for brevity. Since \eqref{eq_safety_distance_int} shows a linear minimization with a lower bound%
\begin{equation}
d_{safe} = \underset{t \le \epsilon \le \infty}{\max} \left\{ \int_t^{\epsilon} v_e(\tau) - v_l(\tau) d\tau \right\}.
\label{eq_safe_dist_open_form}
\end{equation}%
Both $ v_e(t) $ and $ v_l(t) $ are continuous piecewise linear functions, therefore $ \delta v(t) = v_e(t) - v_l(t) $ is also piecewise linear and%
\begin{equation}
J_{dist}(\epsilon) = \int_{t}^{\epsilon} \delta v(\tau) d\tau 
\end{equation}%
is a twice-differentiable piecewise quadratic function. It is also worth noting that%
\begin{equation}
J_{dist}(\epsilon) = J_{dist}(\infty)\;, \forall \epsilon > t + t_f^{max}.
\end{equation}

Since the maximization problem is not always concave but has a small number of possible optimal solutions, the optimization can be expressed analytically through enumeration \cite{lofberg2003minimax}. Three possible positions for the constrained global maximum exists: The lower optimization bound, the upper optimization bound and a local maximum inside the feasible set. The lower bound case can be trivially obtained as $ d_{safe}^{lb} = 0 $, while the upper bound case is%
\begin{equation}
\begin{matrix*}[l]
d_{safe}^{ub} &= J_{dist}(\infty) \\
&= \int_t^{\infty} \delta v(\tau) d\tau \\
&= v_e(t) \phi + \frac{v_e(t)^2}{a_e^b} - \frac{v_l(t)^2}{a_l^b}.
\end{matrix*}
\end{equation}%

The latter case requires the enumeration  of all $ \{\epsilon \mid \dot{J}_{dist}(\epsilon) = 0,\;\ddot{J}_{dist}(\epsilon) < 0\}$. The Hessian of $ J_{dist}(\epsilon) $ can be expressed as%
\begin{equation}
\ddot{J}_{dist}(\epsilon)= \left\{ \begin{matrix*}[l]
a_l^b &, 0 \le \epsilon - t < \phi\\
a_l^b - a_e^b &,\phi \le \epsilon - t < t_f^{min}\\
a_l^b &, t_f^{min} \le \epsilon - t < t_f^{max}, \;t_f^l > t_f^e\\
- a_e^b &, t_f^{min} \le \epsilon - t < t_f^{max}, \; t_f^e >= t_f^l\\
0 &, t_f^{max} < \epsilon - t
\end{matrix*} \right.
\end{equation}%
which it is negative only for $ \epsilon \in \{\epsilon \mid \epsilon \in [t + \phi, t + t_f^{min}),\;a_e^b > a_l^b \} $ or $ \epsilon \in \{\epsilon \mid \epsilon \in [t + t_f^{min}, t + t_f^{max}),\; t_f^e >= t_f^l \} $. However, $ \dot{J}_{dist}(\epsilon) \neq 0\; \forall \epsilon \in [t + t_f^{min}, t + t_f^{max}) $, therefore $ \ddot{J}_{dist}(\epsilon) $ has a local maximum in $ \epsilon^* $ only if $ a_e^b > a_l^b $ and $ \dot{J}_{dist}(\epsilon^*) = 0 $, which can be expanded as
\begin{equation}
v_e(t) - a_e^b (\epsilon^* - \phi) = v_l(t) - a_l^b \epsilon^*
\end{equation}%
and results%
\begin{equation}
\epsilon^* = \frac{v_e(t) - v_l(t) + a_e^b \phi}{a_e^b - a_l^b} \iff \epsilon^* \in [t + \phi, t + t_f^{min})
\end{equation}
The local optima minimum safety distance can be expressed as%
\begin{equation}
\begin{matrix*}[l]
d_{safe}^{lo} &= J_{dist}(\epsilon^*) \\
&= \int_t^{\epsilon^*} \delta v(\tau) d\tau \\
&= (v_e(t) - v_l(t)) \epsilon^* - (a_e^b - a_l^b) \frac{\epsilon^{*2}}{2} + a_e^b \frac{\phi^2}{2}
\end{matrix*}.
\end{equation}

Finally $ d_{safe} $ is the enumeration of all previous possible maximums, which results%
\begin{equation}
d_{safe} = \left\{ \begin{matrix*}[l]
d_{safe}^{lo} &, \epsilon \in [t + \phi, t + t_f^{min}) \\ 
\max\left(0, d_{safe}^{ub} \right)  &, otherwise. 
\end{matrix*}\right.
\end{equation}
\end{proof}

The safety distance concept presented in Theorem \ref{the_safe_distance} is conservative since it does not take into account any other vehicle besides the immediately preceding one. However, it is also the most robust since it guarantees not only the platoon safety, but also the safety of any of its contiguous subsets.

\subsection{Impact of Minimum Safety Distance on Cooperative Adaptive Cruise Control Performance}

Consider both vehicle velocities are equal and in steady steady-state, $ a_e^b = a_l^b = 9 m/s^2 $ and $ \phi = 0.27s $. The minimum safe distance is $ 9.72 m $ for $ v_e(t) = v_l(t) = 35 m/s $ ($126 km/h$) and $ 6.94 m $ for $ v_e(t) = v_l(t) = 25 m/s $ ($90km/h$), as shown in Figure \ref{fig_minimum_distance_curve}.

\begin{figure}[ht]
\centering
\includegraphics[width=\columnwidth]{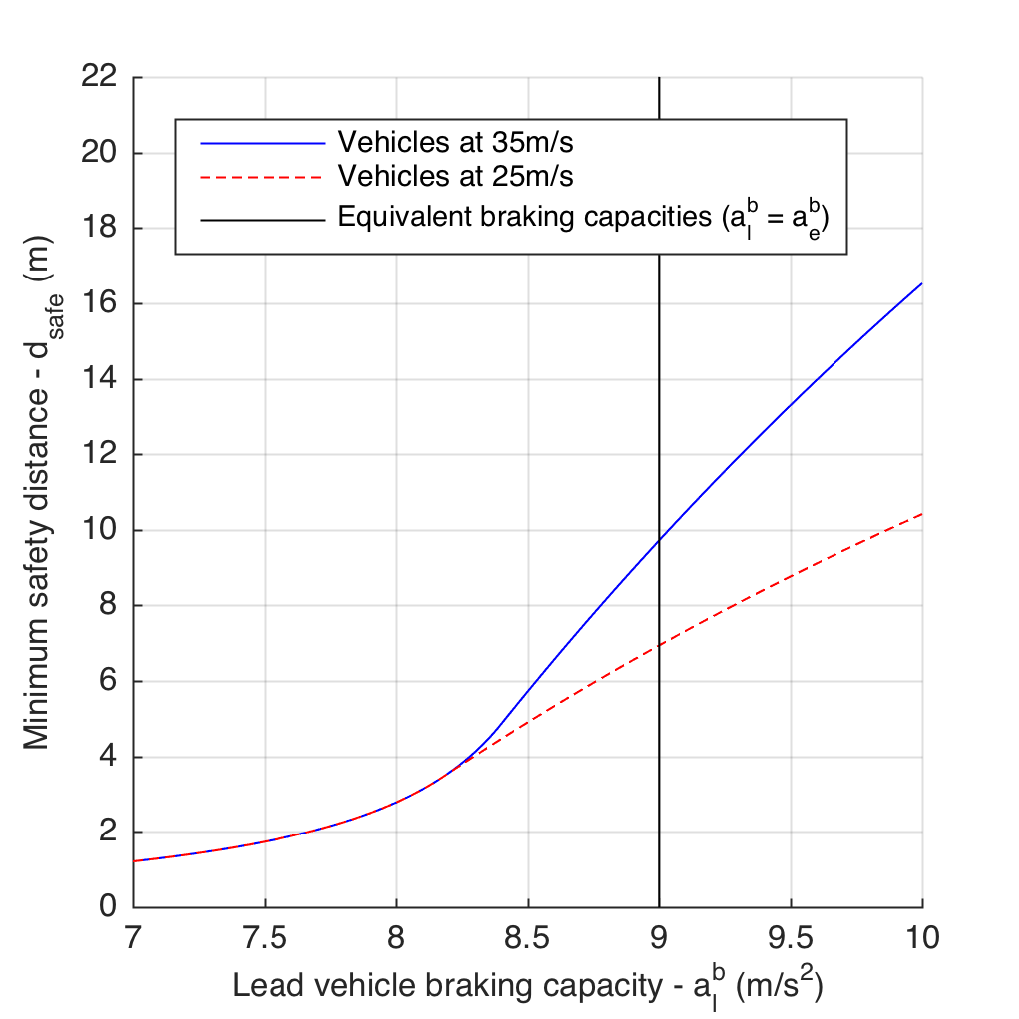}
\caption{The minimum distance in function of the lead vehicle braking capacity for a ego vehicle braking capacity $ a_e^b = 9 m/s^2 $ and both vehicles cruising at $ 35 m/s $ (blue line) and $ 25 m/s $ (red dashed line)}
\label{fig_minimum_distance_curve}
\end{figure}

The minimum safety distance obtained is superior to the distances investigated in previous studies (such as \cite{6093130} and \cite{tsugawa2014results}), where vehicles operated down to $ 2m $ clearances at highway speeds. For the achievement of such clearance, the overall delay required for vehicles with similar braking capacity would be $ 80 ms $ for $ 25m/s $ and $ 57 ms $ for $ 35 m/s $. The required delays are not consistent with state-of-the-art actuators and wireless transmission systems, since current DSRC systems have a worst-case delay of $ 22ms $ on benchmarks \cite{5940479} and braking systems take up to $ 100 ms $ to achieve the commanded pressure. However, there are still significant benefits for CACC systems capable of operating near the minimum safety distance, since \cite{tsugawa2014results} presented an average fuel saving of $13\%$ for heavy trucks at $10m$ clearances.

Sorting the vehicles within the platoon from the least braking capacity (e.g. a heavy truck with a trailer) to the highest braking capacity (e.g. a sport car) would decrease the overall platoon clearance and maintain safety guarantees. However, such sorting would result in situations where the vehicle with the slowest braking response becomes the platoon leader and is entitled of the mitigation of all forward external emergencies. Therefore, there is a trade-off between the minimization of the clearance inside a platoon and the robustness guarantees against external incidents.

\subsection{Safety and Comfort Constraints}

The following two factors must be taken into consideration when techniques for any type of vehicle control are investigated: The vehicle must be safe at all times, and its behavior must be comfortable whenever possible.

Safety constraints must be met at all times, regardless of system disturbances, while comfort must be met whenever it does not compromise the safety of the vehicle. Therefore, four constraint sets were defined:

\begin{itemize}
\item Minimum safety distance set (safety);
\item Time-to-contact set (comfort);
\item Road speed limit set (safety),
\item Acceleration limits set (comfort).
\end{itemize}

\begin{figure}[ht]
\centering
\includegraphics[width=\columnwidth]{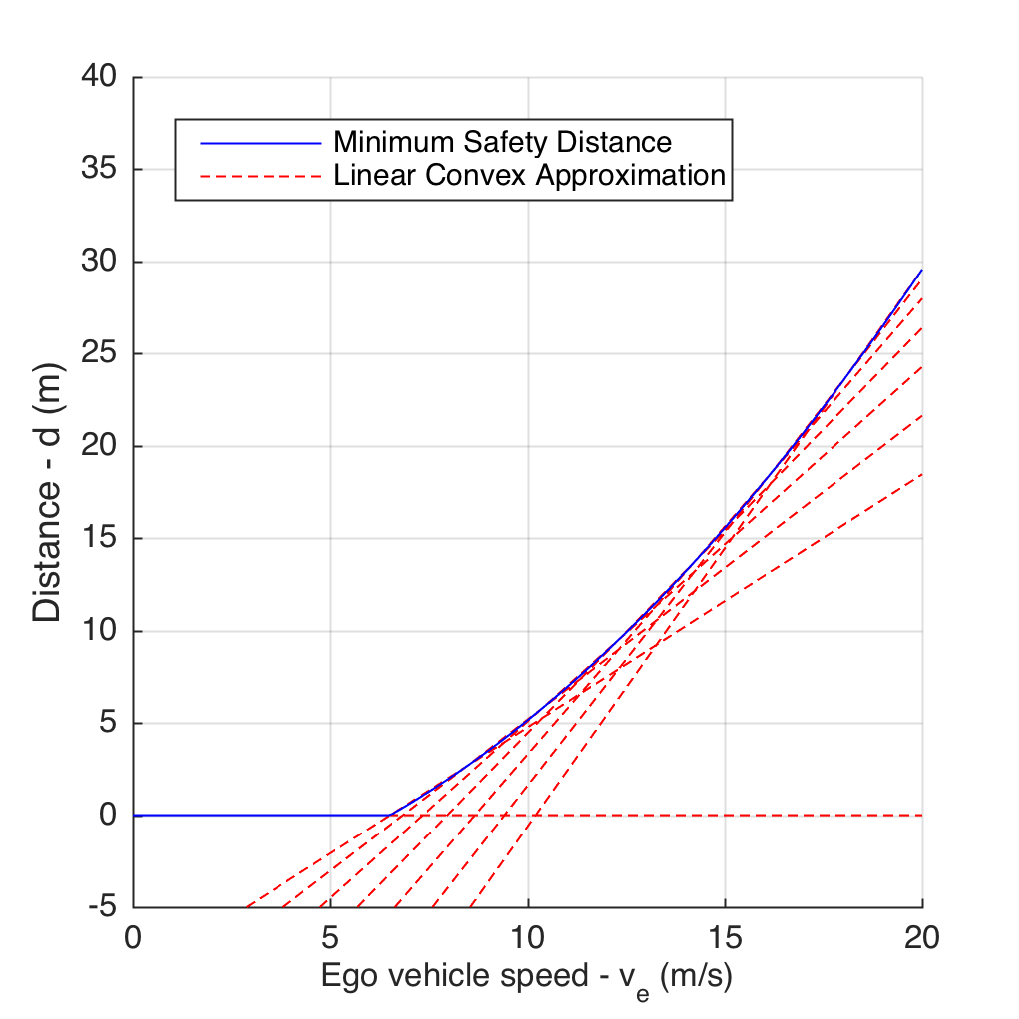}
\caption{Minimum safety distance constraint and its linear convex approximation for $ v_l = 10 m/s $, $ a_e^b = 7 m/s^2 $ and $ a_l^b = 10 m/s^2 $ }
\label{fig_safety_distance_constraint}
\end{figure}

The minimum safety distance set enforces the minimum distance to the lead vehicle, as defined in \eqref{eq_minimum_safety_distance}. This constraint is nonlinear, but convex on $ d(t) $ and $ v_e(t) $ and non-convex on the uncontrollable $ v_l(t) $. 

Let $ f\left(v_e(t), v_l(t)\right) : \mathds{R}^2 \rightarrow \mathds{R} $ define the minimum safety distance relation such that $ - d(t) + f\left(v_e(t), v_l(t)\right) \le 0 $ defines its constraint. A first0-order Taylor approximation for the lead vehicle speed $ v_l(t) \approx v_l(0) + t * a_l(0) $ yields%
\begin{equation}
g\left(t, v_e(t)\right) = f\left(v_e(t), v_l(0) + t * a_l(0) \right)
\end{equation}%
where $ g $ is a convex function that can be conservatively represented by a set of linear constraints, as shown in Figure \ref{fig_safety_distance_constraint}.

This set of linear constraints consists of one constraint for the linear region%
\begin{equation}
 - d(t) \le 0, \forall v_e(t) \in [0, v_{min}]
\end{equation}%
where%
\begin{equation}
v_{min} = \sqrt{\frac{v_l(t)^2 a_e^b}{a_l^b}}
\end{equation}%
and seven constraints between $ v_{min} $ and road speed limit $ v_{max} $ for the nonlinear region given by Taylor approximation of function $ g $ at a given point $ p_i = v_{min} + i (v_{max} - v_{min}) / 7, \forall i \in [0,7] $ %
\begin{equation}
- d(t) + g(t, p_i) + \frac{g(t, p_{i+1}) - g(t, p_{i})}{p_{i+1} - p_i} (v_e(t) - p_i) \le 0.
\end{equation}%
Since this constraint is dependent on $ t $, it must be defined once for each discrete step where the system is evaluated. The set will be compactly denoted as%
\begin{equation}
- \mathds{1}_{8} d_k + f_k + g_k v_{e,k} \le 0
\label{eq_constraint_1}
\end{equation}%
where $ d_k $ and $ v_{e,k} $ are the discrete counterparts of $ d(t) $ and $ v_e(t) $, respectively and $ \mathds{1}_{i} $ is an one-column vector of size $ i $.

The time-to-contact constraint avoids fast approximations to the lead vehicle that might be uncomfortable to the passengers.

\begin{definition}
Let time-to-contact $ t_c(t) $ be the instantaneous time required for a collision to occur between the lead and ego vehicle if no vehicle varies its speed.%
\begin{equation}
\begin{matrix*}[l]
t_c(t) =& \min & \tau\\
& s.t. & d(\tau) = 0\\
& & \dot{d}(\tau) = v_l(t) - v_e(t)\\
& & \tau \ge 0
\end{matrix*}
\end{equation}
\label{def_time_to_contact}%
which has a closed form solution%
\begin{equation}
t_c(t) = \max\left(\frac{d(t)}{v_{e}(t) - v_{l}(t)}, 0 \right).
\end{equation}
\end{definition}

\begin{assumption}
A minimum value of $ t_c $, denoted by $ t_{c,min} $, is considered comfortable by the driver and passengers of the ego vehicle.
\label{ass_min_ttc}
\end{assumption}

The time-to-contact constraint based on Definition \ref{def_time_to_contact} and Assumption \ref{ass_min_ttc} is%
\begin{equation}
t_{c,min} \le t_c(t) = \frac{d(i)}{v_{e}(i) - v_{l}(i)},\;\forall t \in \mathds{R}
\end{equation}%
which can be rewritten in the discrete canonical form as%
\begin{equation}
-d_k + t_{c,min} (v_{e,k} - v_{l,k}) \le 0.
\label{eq_constraint_2}
\end{equation}

The road speed limit constraint set avoids the vehicle exceeding the maximum speed on the road. Therefore, it can be trivially defined as%
\begin{equation}
0 \le v_{e,k} \le v_{max}.
\end{equation}

Finally, the acceleration constraint set limits the controller acceleration to what is achievable given the current friction limits and brake distribution and limits the vehicle acceleration to a comfortable region whenever possible. For the first role, the constraint can be defined in the canonical form as%
\begin{equation}
- a_{e,k} \le - a_e^b
\end{equation}%
while for the second, a slack variable $ s_{a,k} $ will be introduced for the creation of a $soft$ constraint, which results%
\begin{equation}
a_{min} \le a_{e,k} + s_{a,k} \le a_{max}
\end{equation}%
where $ a_{min} $ and $ a_{max} $ define the comfortable acceleration limits and%
\begin{equation}
x_k \leftarrow [x_k^T, s_{a,k}]^T
\end{equation}%
for optimization purposes.

All constraints presented will be represented by their generic form $ A_k x_k + B_k u_k + c_k \le 0 $, where%
\begin{equation}
\begin{matrix}
A_k = \begin{bmatrix*}[r]
0 & 0 & -1 & 0\\
0 & 0 & 1 & 0\\
-1 & t_{c,min} & -t_{c,min} & 0\\
0 & 0 & 0 & 0\\
0 & 0 & 0 & -1\\
0 & 0 & 0 & 1\\
-\mathds{1}_8 & 0 & g_k & 0
\end{bmatrix*}
\\ \\ B_k = \begin{bmatrix*}[r]
0\\
0\\
0\\
-1\\
-1\\
1\\
0
\end{bmatrix*}
\; c_k = \begin{bmatrix*}[r]
0\\
v_{max}\\
0\\
- a_e^b\\
- a_{min}\\
a_{max}\\
- f_k
\end{bmatrix*}.
\end{matrix}
\end{equation}%
\section{Proposed controller}
\label{sec_controller}

This Section addresses the formulation of the $l\infty $-norm robust optimal control based on the min-max approach and presents the pre-stabilizing nil-potent controller concept.

\subsection{$ l\infty $-norm Robust Optimal Control Formulation}

The optimization problem of an $ l\infty $-norm optimal receding horizon control formulation \cite{christophersen2006optimal} is given by%
\begin{equation}
\begin{matrix*}[l]
J^*(x_0) =& \underset{x,u}{\min} & \overset{T - 1}{\underset{k = 0}{\sum}} || Q x_k ||_{\infty} + || R u_k ||_{\infty} + || P_\infty x_{T} ||_{\infty}\\
& s.t. & x_{k+1} = F x_{k} + G u_{k} + h + W w_{k}\\
& & A_k x_{k} + B_k u_{k} + c_k \le 0\\
\end{matrix*}
\label{eq_opt_infty_norm}
\end{equation}%
where the range of $ k $ has been omitted for brevity, $T$ is the constrained horizon, $F \in \mathds{R}^{N_x \times N_x}$, $G \in \mathds{R}^{N_x \times N_u}$ and $h \in \mathds{R}^{N_x}$ represent the dynamics of an affine system, $W \in \mathds{R}^{N_x \times N_w}$ is the additive disturbance matrix, $A_k \in \mathds{R}^{N_c \times N_x}, \forall k \in [0,T] $, $B_k \in \mathds{R}^{N_c \times N_u}, \forall k \in [0,T]$ and $c_k \in \mathds{R}^{N_c}, \forall k \in [0,T]$ define a polytopic inequality constraint, $ Q \in \mathds{R}^{N_wx \times N_x} $ and $ R \in \mathds{R}^{N_wc \times N_u} $ are full lower rank real matrices and $ P_\infty \in \mathds{R}^{N_p \times N_x} $ is the cost matrix of the infinite horizon unconstrained $ l\infty $-norm problem%
\begin{equation}
\begin{matrix*}[l]
|| P_\infty x_{0} ||_{\infty} =& \underset{x,u}{\min} & \overset{\infty}{\underset{k = 0}{\sum}} || Q x_k ||_{\infty} + || R u_k ||_{\infty}\\
& s.t. & x_{k+1} = F x_{k} + G u_{k} + h
\end{matrix*}
\end{equation}

The problem in \eqref{eq_opt_infty_norm} can be represented as a Linear Programming (LP) problem based on a relaxation of the $l\infty$-norm cost functional through the addition of auxiliary variables $ \epsilon^x_k \in \mathds{R},\; \forall k \in [0,T] $ and $ \epsilon^u_k \in \mathds{R},\; \forall k \in [0,T-1] $, which yields%
\begin{equation}
\begin{matrix*}[l]
J^*(x_0) =& \underset{x,u}{\min} & \epsilon^x_T + \overset{T - 1}{\underset{k = 0}{\sum}} \epsilon^x_k + \epsilon^u_k\\
&\begin{matrix*}[l]
s.t.\\
\phantom{-}\\
\phantom{-}\\
\phantom{-}\\
\phantom{-}\\
\phantom{-}\\
\phantom{-}\\
\phantom{-}\\
\end{matrix*} 
&
\begin{matrix*}[l]
x_{k+1} = F x_{k} + G u_{k} + h + W w_{k}\\
A_k x_{k} + B_k u_{k} + c_k \le 0\\
Q_i x_k \le \epsilon^x_k\\
- Q_i x_k \le \epsilon^x_k\\
R_i u_k \le \epsilon^u_k\\
-R_i u_k \le \epsilon^u_k\\
(P_\infty)_i x_T \le \epsilon^x_T\\
-(P_\infty)_i x_T \le \epsilon^x_T\\
\end{matrix*}
\end{matrix*}
\label{eq_opt_infty_norm_lp}
\end{equation}%
where $ Q_i $, $ R_i $ and $ (P_\infty)_i $ are the $i$-th row of matrices $ Q $, $ R $ and $ P_\infty $, respectively.

The system should remain constrained and stable under any possible value of $ w \in \mathcal{W} $, where $ \mathcal{W} = \{ w \mid w \in \mathds{R}^{N}, \;|| w ||_\infty \le 1 \}$. Therefore, let $ \bar{x}_k \in \mathds{R}^{N_x} $ be the process state without disturbances. The representation of both $ x_{k} $ and $ \bar{x}_k $ in respect to a known current state $ x_0 = \bar{x}_0 $ results in%
\begin{equation}
\begin{matrix*}[l]
x_k &= F^{k-1} x_{0} + \overset{k - 1}{\underset{i = 0}{\sum}} F^{k - i - 1} (G u_i + h) + \overset{k - 1}{\underset{i = 0}{\sum}} F^{k - i - 1} W w_i, \\
\bar{x}_k &= F^{k-1} x_{0} + \overset{k - 1}{\underset{i = 0}{\sum}} F^{k - i - 1} (G u_i + h)
\end{matrix*}
\label{eq_certain_x}
\end{equation}%
and defining $ \bar{W}_{i} = F^{i-1} W $ for brevity, the relation between $ x_k $ and $ \bar{x}_k $ is%
\begin{equation}
x_k = \bar{x}_k + \overset{k - 1}{\underset{i = 0}{\sum}} \bar{W}_{k-i} w_i.
\label{eq_certain_x_uncertain_x}
\end{equation}%
The substitution of \eqref{eq_certain_x_uncertain_x} in \eqref{eq_opt_infty_norm_lp} yields%
\begin{equation}
\begin{matrix*}[l]
J^*(x_0) =& \underset{x,u}{\min} & \epsilon^x_T + \overset{T - 1}{\underset{k = 0}{\sum}} \epsilon^x_k + \epsilon^u_k\\
&\begin{matrix*}[l]
s.t.\\
\phantom{-}\\
\phantom{\overset{k - 1}{\underset{i = 0}{\sum}}}\\
\phantom{\overset{k - 1}{\underset{i = 0}{\sum}}}\\
\phantom{\overset{k - 1}{\underset{i = 0}{\sum}}}\\
\phantom{-}\\
\phantom{\overset{k - 1}{\underset{i = 0}{\sum}}}\\
\phantom{\overset{k - 1}{\underset{i = 0}{\sum}}}\\
\end{matrix*} 
& 
\begin{matrix*}[l]
\bar{x}_{k+1} = F x_{k} + G u_{k} + h\\
A_k \bar{x}_{k} + B_k u_{k} + c_k + A \overset{k - 1}{\underset{i = 0}{\sum}} \bar{W}_{k-i} w_i \le 0\\
Q_i \bar{x}_k + Q_i \overset{k - 1}{\underset{j = 0}{\sum}} \bar{W}_{k-j} w_j \le \epsilon^x_k\\
- Q_i \bar{x}_k - Q_i \overset{k - 1}{\underset{j = 0}{\sum}} \bar{W}_{k-j} w_j \le \epsilon^x_k\\
R_i u_k \le \epsilon^u_k\\
-R_i u_k \le \epsilon^u_k\\
(P_\infty)_i \bar{x}_T + (P_\infty)_i \overset{T - 1}{\underset{j = 0}{\sum}} \bar{W}_{k-j} w_j \le \epsilon^x_T\\
-(P_\infty)_i \bar{x}_T - (P_\infty)_i \overset{T - 1}{\underset{j = 0}{\sum}} \bar{W}_{k-j} w_j \le \epsilon^x_T\\
\end{matrix*}
\end{matrix*}
\label{eq_opt_infty_norm_lp_2}
\end{equation}%
where all uncertain variables are only present on polytopic constraints.

\begin{theorem}
Let $ p\in \mathds{R}_{+\infty} $, $ p^* \in \mathds{R}_{+\infty} $, $ q \in \mathds{R}^{N} $ and $ x \in \mathcal{X} $ where $ \mathcal{X} = \{ x \mid x \in \mathds{R}^{N}, \;|| x ||_p \le 1 \}$ and $ 1/p + 1/p^* = 1 $ then%
\begin{equation}
q^T x \le \underset{||x||_p \le 1}{\max} q^T x = || q ||_{p^{*}}.
\end{equation}
\label{theorem_complimentary_norm}
\end{theorem}
\begin{proof}
Adapted from \cite{boyd2004convex}. If both $ q $ and $ x $ are finite, there is an upper bound to $ q^T x $ given by Hölder's inequality as%
\begin{equation}
q^T x \le ||x||_p ||q||_{p^*}
\end{equation}%
by definition, $ ||x||_p \le 1 $, therefore%
\begin{equation}
q^T x \le ||x||_p ||q||_{p^*} \le ||q||_{p^*}
\label{eq_holders_results}
\end{equation}%
results in an upper bound to the constrained maximization
\begin{equation}
J(q) = \underset{||x||_p \le 1}{\max} q^T x = ||q||_{p^*}.
\end{equation}
\end{proof}

Let $ |M| $ be the element-wise modulus of a given matrix M. Through Theorem \ref{theorem_complimentary_norm}, it follows that
\begin{equation}
M w \le \underset{||w||_\infty \le 1}{\max} M w = | M | \mathds{1}_{N_w}
\end{equation}%
and the robust polytopic constraint from \eqref{eq_opt_infty_norm_lp_2} is
\begin{equation}
A_k \bar{x}_{k} + B_k u_{k} + c_k + \overset{k - 1}{\underset{i = 0}{\sum}} | A_k \bar{W}_{k-i} | \mathds{1}_{N_w} \le 0
\label{eq_robust_polytopic}
\end{equation}%
and analogously%
\begin{equation}
\begin{matrix*}[l]
Q_i \bar{x}_k + \overset{k - 1}{\underset{i = 0}{\sum}} | Q_i \bar{W}_{k-i} | \mathds{1}_{N_w} \le \epsilon^x_k\\
\le - Q_i \bar{x}_k + \overset{k - 1}{\underset{i = 0}{\sum}} | Q_i \bar{W}_{k-i} | \mathds{1}_{N_w} \le \epsilon^x_k\\
(P_\infty)_i \bar{x}_T + \overset{T - 1}{\underset{j = 0}{\sum}} | (P_\infty)_i \bar{W}_{k-j}|  \mathds{1}_{N_w} \le \epsilon^x_T\\
- (P_\infty)_i \bar{x}_T + \overset{T - 1}{\underset{j = 0}{\sum}} | (P_\infty)_i \bar{W}_{k-j}|  \mathds{1}_{N_w} \le \epsilon^x_T.\\
\end{matrix*}
\label{eq_robust_polytopic_2}
\end{equation}

The robust counterpart of the optimization problem from \eqref{eq_opt_infty_norm_lp_2} based on \eqref{eq_robust_polytopic} and \eqref{eq_robust_polytopic_2} is%
\begin{equation}
\begin{matrix*}[l]
\bar{J}^*(x_0) =& \underset{x,u}{\min} & \epsilon^x_T + \overset{T - 1}{\underset{k = 0}{\sum}} \epsilon^x_k + \epsilon^u_k\\
&\begin{matrix*}[l]
s.t.\\
\phantom{-}\\
\phantom{-}\\
\phantom{-}\\
\phantom{-}\\
\phantom{-}\\
\phantom{-}\\
\phantom{-}\\
\end{matrix*} 
& 
\begin{matrix*}[l]
\bar{x}_{k+1} = F x_{k} + G u_{k} + h\\
A_k \bar{x}_{k} + B_k u_{k} + c_k + \Phi_A(k) \le 0\\
Q_i \bar{x}_k + \Phi_Q(k) \le \epsilon^x_k\\
- Q_i \bar{x}_k + \Phi_Q(k) \le \epsilon^x_k\\
R_i u_k \le \epsilon^u_k\\
-R_i u_k \le \epsilon^u_k\\
(P_\infty)_i \bar{x}_T + \Phi_P(T) \le \epsilon^x_T\\
-(P_\infty)_i \bar{x}_T + \Phi_P(T) \le \epsilon^x_T\\
\end{matrix*}
\end{matrix*}
\label{eq_opt_infty_norm_lp_3}
\end{equation}%
such that $ \bar{J}^*(x_0) \ge J^*(x_0) $ and where $ \Phi_A(k) = \overset{k - 1}{\underset{i = 0}{\sum}} | A_k \bar{W}_{k-i} | \mathds{1}_{N_w} $, $ \Phi_Q(k) = \overset{k - 1}{\underset{i = 0}{\sum}} | Q_i \bar{W}_{k-i} | \mathds{1}_{N_w}$ and $ \phi_P(k) = \overset{k - 1}{\underset{j = 0}{\sum}} | (P_\infty)_i \bar{W}_{k-j}| \mathds{1}_{N_w} $.

\subsection{Pre-Stabilizing Nil-Potent Controller}

Due to the disturbance modeled in the system, the formulation presented in \eqref{eq_opt_infty_norm_lp_3} has the following disadvantage: the functional cost $ \underset{T \rightarrow \infty}{\lim} \bar{J}^* \rightarrow \infty $ and the feasible set $ \mathcal{C}_k = \{ x_k,u_k \mid A x_k + B u_k + \bar{c}_k \le 0 \} $ become empty as $ k \rightarrow \infty $ if the open-loop system is unstable or marginally stable. A possible approach for the mitigation of such an effect is the use of a nil-potent controller \cite{bemporad1999robust}.

\begin{definition}
\label{def_nil_potent}
Let $ x_k \in \mathds{R}^{N_x} $ be the states and $ u_k \in \mathds{R}^{N_u} $ the control inputs of a system defined by $ F \in \mathds{R}^{N_x \times N_x} $, $ G \in \mathds{R}^{N_x \times N_u} $, such that $ x_{k+1} = F x_k + G u_k $. A controller $ K_0 \in \mathds{R}^{N_u \times N_x} $ is an $i$-th order nil-potent controller if%
\begin{equation}
(F - G K_0)^n = \mathds{O}_{N_x}\;,\; n \in \mathds{I}, n > i
\end{equation}%
where $ \mathds{O}_{j} $ is a zero column vector of size $ j $.
\end{definition}

Let $ u_k = - K_0 x_k + v_k $ and $ v_k $ be the new control input. The system pre-stabilized by a $n$-th order nil-potent controller is%
\begin{equation}
\bar{x}_{k+1} = (F - G K_0) \bar{x}_k + G v_k + h
\label{eq_nil_potent_system_certain}
\end{equation}%
and, given the results from \eqref{eq_certain_x_uncertain_x}, the system model with disturbances is written as%
\begin{equation}
x_k = \bar{x}_k + \overset{\min(k,n)}{\underset{i = 1}{\sum}} \bar{W}_{i} w_{k-i}
\label{eq_nil_potent_system_uncertain}
\end{equation}%
since%
\begin{equation}
\bar{W}_i = \mathds{O}_{N_x},\; \forall i > n.
\end{equation}%

Based on \eqref{eq_nil_potent_system_certain}, \eqref{eq_nil_potent_system_uncertain} and \eqref{eq_opt_infty_norm_lp_3}, the pre-stabilized robust $ l\infty $-norm optimal control problem, $\bar{J}_{np}^*(x_0)$ is%
\begin{equation}
\begin{matrix*}[l]
\bar{J}_{np}^*(x_0) =& \underset{x,u}{\min} & \epsilon^x_T + \overset{T - 1}{\underset{k = 0}{\sum}} \epsilon^x_k + \epsilon^u_k\\ 
&\begin{matrix*}[l]
s.t.\\
\phantom{-}\\
\phantom{-}\\
\phantom{-}\\
\phantom{-}\\
\phantom{-}\\
\phantom{-}\\
\phantom{-}\\
\end{matrix*} 
& 
\begin{matrix*}[l]
\bar{x}_{k+1} = \widetilde{F} x_{k} + G v_{k} + h\\
\widetilde{A}_k \bar{x}_{k} + B_k v_{k} + c_k + \Phi_A(k) \le 0\\
Q_i \bar{x}_k + \Phi_Q(k) \le \epsilon^x_k\\
- Q_i \bar{x}_k + \Phi_Q(k) \le \epsilon^x_k\\
- R_i K_0 x_k + R_i v_k \le \epsilon^u_k\\
R_i K_0 x_k - R_i v_k \le \epsilon^u_k\\
(P_\infty)_i \bar{x}_T + \Phi_P(T) \le \epsilon^x_T\\
-(P_\infty)_i \bar{x}_T + \Phi_P(T) \le \epsilon^x_T\\
\end{matrix*}
\end{matrix*}
\label{eq_opt_infty_norm_lp_4}
\end{equation}%
where $ \widetilde{F} = F - G K_0 $ and $ \widetilde{A}_k = A_k - B_k K_0 $.

In this new form $ \mathcal{C}_k \ge \mathcal{C}_n, \forall k \in [0,T] $. Therefore, the feasible set converges to a smaller robust set given by $ \mathcal{C}_n $ instead of becoming empty as $ k \rightarrow \infty $.

\section{Controller Validation}
\label{sec_experiments}

This section is divided into two parts. The first describes the simulation framework developed for the validation of the control law and the second reports the simulation results for both robust and nominal controllers.

\begin{figure}[!ht]
\centering
\includegraphics[width=\columnwidth]{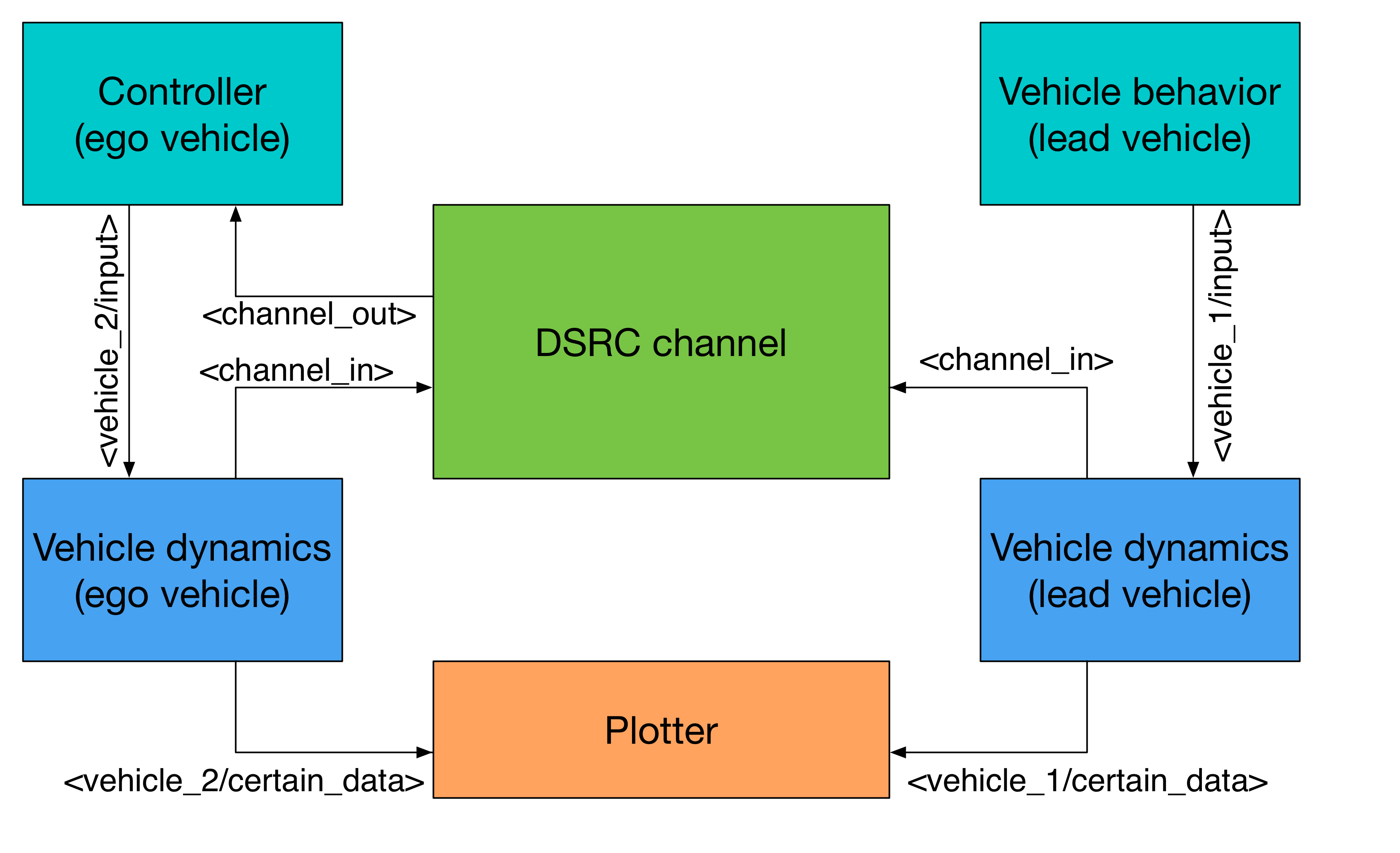}
\caption{Simulation framework architecture}
\label{fig_simulation_arch}
\end{figure}

\setcounter{figure}{5}
\begin{figure}[!b]
\centering
\includegraphics[width=\columnwidth]{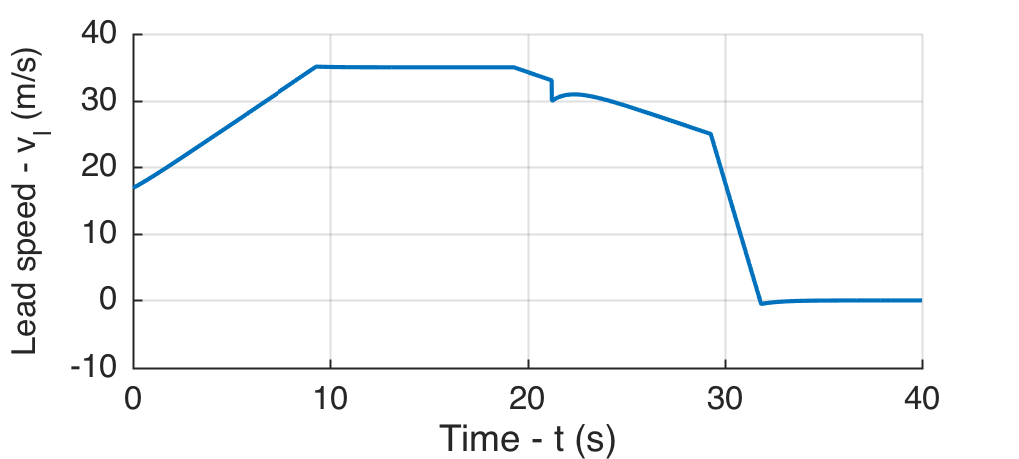}
\caption{Simulated lead vehicle speed profile}
\label{fig_lead_profile}
\end{figure}

\setcounter{figure}{6}
\begin{figure*}[ht]
\centering
\includegraphics[width=\textwidth]{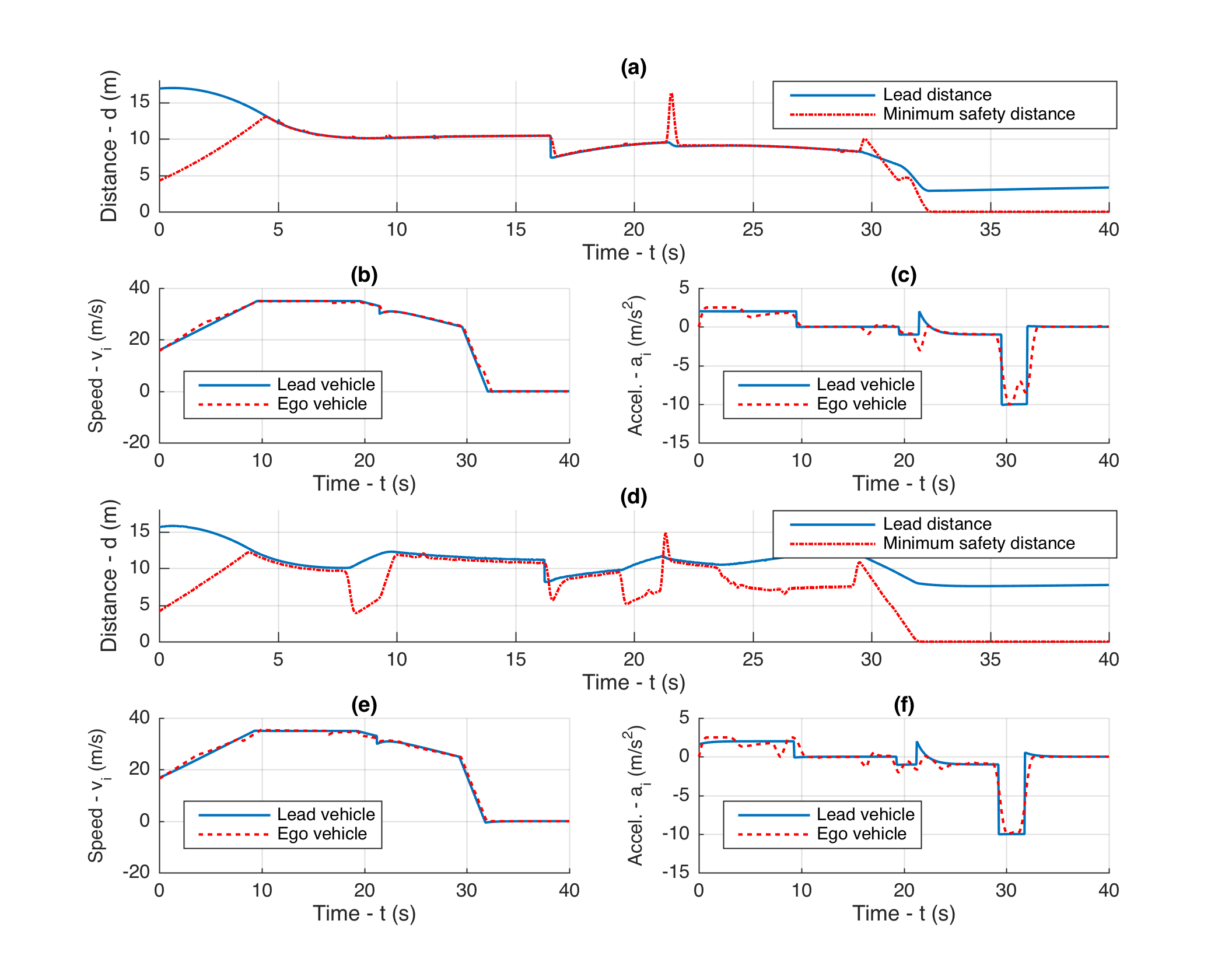}
\caption{Simulation results, (a) and (d) present the distance between both vehicles (blue line) and the minimum safety distance (red dashed line) for the nominal and robust $l\infty$-norm controllers, respectively; (b), (c), (e) and (f) show the lead (blue line) and ego (red dashed line) vehicles speed and acceleration profiles for the nominal and robust controllers, respectively}
\label{fig_simulation_result}
\end{figure*}

\subsection{Simulation Framework}

A simulation framework for multiple communicating vehicles was developed for the validation of the proposed controller based on the Robotic Operating System (ROS) \cite{ros2009}. It consists of five different processes:%
\begin{enumerate}
\item Vehicle dynamics process, which simulates longitudinal and lateral vehicle dynamics and provides noise for state measurements;
\item DSRC communication process, which provides realistic delays and packet losses for V2V communication;
\item Vehicle behavior process, which executes a predefined behavior profile on a vehicle;
\item Controller process, which executes the proposed controller; and
\item Plotter process, which displays the simulation results.
\end{enumerate}

The architecture is shown in Figure \ref{fig_simulation_arch} where blocks represent processes and arrows represent message passing topics between the processes. The simulation assumed a $22ms$ communication delay and a $1\%$ packet loss based on the worst case scenario from \cite{5940479}.

The proposed controller was implemented using CVXGEN \cite{mattingley2012cvxgen}, a code-generation tool for small LP and QP solvers. The LP problem consisted of $80$ optimization variables, $30$ equality constraints and $248$ inequality constraints for a horizon of $10$ time-steps.

\subsection{Simulation results}

A short-term application of CACC systems focus on highway driving scenarios. Therefore, the chosen lead vehicle speed profile, shown in Figure \ref{fig_lead_profile}, starts at $15m/s$ ($54km/h$) performing an on-ramp acceleration of $2m/s^2$ until it reaches $35m/s$ ($126km/h$) at $t=10s$. At this time, it maintains constant speed for $10s$, until $t=20s$. It reduces speed at $-1m/s^2$ until reaching $25m/s$ ($90km/h$) at $t=30s$. At this time, it starts an emergency braking maneuver accelerating at $-10m/s^2$ until it reaches a complete stop at $t=35s$. This profile allows the validation of the proposed controller in both nominal and emergency cases.

Disturbances were added for the investigation of the controller robustness to unmodeled uncertainties. A step of amplitude $-3m$ was added to the distance between the vehicles at $t=17s$ and a step of amplitude $-3m/s$ was added to the lead vehicle speed at $t=22s$. The disturbance amplitudes were chosen to ensure violation of the minimum safety distance for both nominal and robust controllers.

The initial conditions and vehicle parameters used in the simulations are shown in Table \ref{table_simulation}. Weighting matrices $ Q $ and $ R $  are given by%
\begin{equation}
Q = \begin{bmatrix}
100 & 0 & 0\\
0 & 1 & -1
\end{bmatrix},\;\;\;
R = \begin{bmatrix} 
1 
\end{bmatrix}
\end{equation}%
and the uncertainty matrix is given by $ W = \begin{bmatrix} 0 & 1.2 & 0 \end{bmatrix}^T $. The lead vehicle had unbounded jerk while the lead vehicle has actuators modeled by a first-order low pass with time constant $t_c=0.1s$ to ensure worst case performance.

\begin{table}[thpb]
	\centering
	\caption{Initial conditions and vehicle parameters of the simulation}
	\label{table_simulation}
    \renewcommand{\arraystretch}{1.2}
    \begin{tabular}{|lccc|}    
    \hline
    Description  & Symbol  & Value  & Unit \\
    \hline
    Initial distance & $d(0)$ & 15 & $m$ \\
    Initial lead vehicle speed & $v_l(0)$ & 15 & $\frac{m}{s}$ \\
    Initial ego vehicle speed & $v_e(0)$ & 15 & $\frac{m}{s}$ \\
    Maximum allowed speed & $v_{max}$ & 40 & $\frac{m}{s}$ \\
    Maximum acceleration & $a_{max} $ & 2.5 & $\frac{m}{s^2}$ \\
    Minimum acceleration & $a_{min} $ & -2.5 & $\frac{m}{s^2}$ \\
    Minimum time-to-contact & $t_{c,min}$ & 2 & $s$ \\
    Worst-case system delay & $\psi$ & 0.3 & $s$ \\
    Lead vehicle maximum braking & $a_l^b$ & 10 & $\frac{m}{s^2}$ \\
    Ego vehicle maximum braking & $a_e^b$ & 10 & $\frac{m}{s^2}$ \\
    Operation frequency & $-$ & 20 & $Hz$ \\
    \hline
    \end{tabular}
\end{table}

Simulation results for the nominal and robust controllers are shown in Figure \ref{fig_simulation_result}. In Figures \ref{fig_simulation_result}-(a) and \ref{fig_simulation_result}-(d), both controllers the ego vehicle accelerates at the maximum allowed rate in order to catch up and reduce the distance to the vehicle ahead from time $t=0s$ until $t=4s$. Where they keep a constant clearance, close to the minimum safety distance, where the robust controller maintains a higher clearance until $t=10s$.

At $t=10s$, the lead vehicle stops accelerating, so that the ego vehicle keeps a distance close to the minimum until the lead vehicle starts to decelerate at $t=20s$. Notice the performances of both control are similar until $t=17s$, when a distance disturbance was introduced. The robust controller performs better in order to recover a safety distance between the vehicles. When the lead vehicle decelerates at $t=20s$, the robust controller increases the robustness margin due to the prediction of a slower speed ahead. The speed disturbance at time $t=22s$ caused an increase to the minimum safety distance, and both controllers responded in a similar manner. 

The main advantage of the proposed robust approach is observed at the emergency braking maneuver at $t=30s$, where it did not violate the minimum distance unlike the nominal controller. Since its model internalized lead vehicle acceleration uncertainties.


\section{Conclusion}
\label{sec_conclusion}

This paper proposed an analytical formulation of the minimum safe distance between two vehicles and a robust $l\infty$-norm optimal controller for cooperative adaptive cruise control based on such minimum distance. The formulation provides a theoretical lower bound to vehicle clearance inside platoons and was used for the design of a controller capable of internalizing lead vehicle acceleration behavior uncertainties. It guarantees the ego vehicle never collides with its preceding vehicle if the preceding vehicle also does not collide.

Validation simulations included on-ramp accelerations, small velocity variations and emergency braking situations, typical of highway driving. The results demonstrate that the proposed robust controller operates correctly in both nominal and emergency scenarios. It was able to achieve a good distance and speed tracking and even increased the vehicle clearance during breaking to ensure the vehicle safety, resulting in a controller with intuitive behavior and robust to the lead vehicle actions.

Future work on the controller will focus on reductions in acceleration disturbances due to measurement noises.


\bibliographystyle{IEEEtranS}
\bibliography{refs}

\begin{IEEEbiography}[{\includegraphics[width=1in,height=1.25in,clip,keepaspectratio]{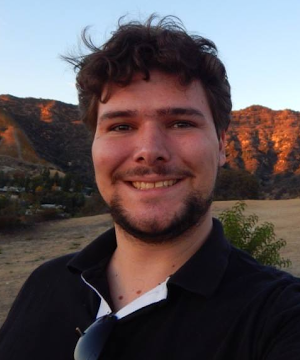}}]{Carlos Massera Filho} is currently working towards his Ph.D. degree at the institute of Mathematics and Computer Science at the University of São Paulo. He received his B.Sc. in Computer Engineering from São Carlos School of Engineer in 2012. He has been working on autonomous and cooperative vehicle control and estimation systems and his current research interests are constrained optimal control, robust optimal control, robust convex optimization and its applications to autonomous and cooperative vehicles.
\end{IEEEbiography}

\begin{IEEEbiography}[{\includegraphics[width=1in,height=1.25in,clip,keepaspectratio]{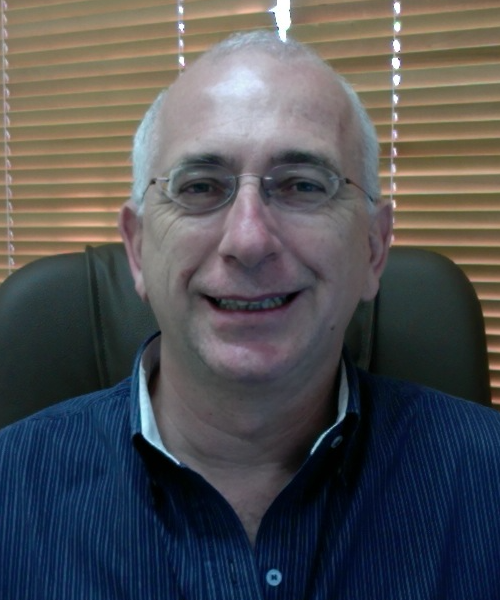}}]{Marco H. Terra}
is Full Professor of Electrical Engineering at University of São Paulo (USP) at São Carlos, Brazil. He received his Ph.D. in Electrical Engineering in 1995 from USP. He has reviewed papers for more than thirty journals and to the Mathematical Reviews of the American Mathematical Society. He is an ad hoc referee for the Research Grants Council (RGB) of Hong Kong. He has published more than 190 conference and journal papers. He is co-author of the book “Robust Control of Robots: Fault Tolerant Approaches” published by Springer. He was coordinator of the robotics committee and president of the Brazilian Automation Society. His research interests include filtering and control theories, fault detection and isolation problems, and robotics.
\end{IEEEbiography}

\begin{IEEEbiography}[{\includegraphics[width=1in,height=1.25in,clip,keepaspectratio]{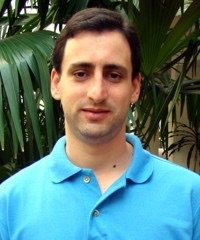}}]{Denis F. Wolf}
is an Associate Professor in the Department of Computer Systems from the University of São Paulo (USP). 
He obtained his PhD degree in Computer Science from the University of Southern California USC in 2006. 
He is the current Director of the Mobile Robotics Laboratory at USP and a member of the Center for Robotics at USP. 
His current research interests include Mobile Robotics, Intelligent Transportation Systems, Machine Learning and Computer Vision.
\end{IEEEbiography}
\end{document}